\documentclass[a4paper,12pt]{article}

\usepackage{color}
\usepackage{amsmath,amssymb,amsthm}
\usepackage{latexsym}
\usepackage{graphicx}
\usepackage{epsfig}
\usepackage{float,enumerate,array}
\usepackage[left=33mm, right=33mm,
            top=30mm, bottom=30mm]{geometry}
\newtheorem{theo}{Theorem}
\newtheorem{prop}{Proposition}
\newtheorem{lem}{Lemma}
\newtheorem{definition}{Definition}
\newtheorem*{remark}{Remark}
\newtheorem{coro}{Corollary}

\newcommand{\bbR}{\mathbb{R}}
\newcommand{\bbZ}{\mathbb{Z}}
\newcommand{\set}[1]{\left\{#1\right\}}
\newcommand{\eps}{\varepsilon}
\newcommand{\lf}{\lfloor}
\newcommand{\rf}{\rfloor}

\newcommand{\card}{\mathrm{card}}

\newcommand{\Z}{\mathbb{Z}}
\newcommand{\Cross}{\Z^{\times}}

\newcommand{\Zero}{\mathbf{0}}

\title{The shape of large balls in highly supercritical percolation}
\author{A-L. Basdevant, N. Enriquez, L. Gerin, J-B. Gou\'er\'e}

\begin{document}
\maketitle
\begin{abstract}
We exploit a connection between distances in the infinite percolation cluster, when the parameter is close to one, and the discrete-time TASEP on $\mathbb{Z}$. This shows that when the parameter goes to one, large balls in the cluster are asymptotically shaped near the axes like arcs of parabola.
\end{abstract}


\section{Introduction}

Our main issue in this paper is to describe the shape of large balls in the infinite two-dimensional percolation cluster, when the percolation parameter is close to one. This problem is closely related to first-passage percolation, a model introduced in the 60's by Hammersley and Welsh \cite{HW} in which one estimates the minimal distance $D(\Zero,x)$ between the origin $\Zero$ and a given point $x$ of $\Z^2$, when edges have i.i.d. positive finite lengths. Distances in the cluster correspond in this framework to the extreme case where edges have lengths $1$ with probability $p\in(0,1)$ and infinite length with probability $1-p$. We refer to \cite{Blair} for a recent survey on first passage percolation and shape theorems.


In first passage percolation, Kingman's subadditive ergodic theorem is the crucial tool to study the asymptotics of distances between distant points.
Garet and Marchand (\cite{GM}, Th.3.2) adapted this argument to the case where edges may have infinite length. They proved the existence, for all $z$ in $\bbR^2$, of a constant $\mu(z)$ such that, if we denote by $[nz]$ one of the closest lattice points to $nz$,  on the event\footnote{We write $v\leftrightarrow w$ if $v,w\in\Z^2$ belong to the same connected component, and $v\leftrightarrow \infty$ if $v$ is in the infinite cluster.} $\set{\Zero\leftrightarrow \infty}$ and along the subsequence
$\set{\Zero\leftrightarrow [nz]}$ , we have a.s.
\begin{equation}\label{Eq:mu}
\lim_{\substack{n \to \infty \\\Zero\leftrightarrow [nz]}} \frac{D(\Zero,[nz])}{n}=\mu(z).
\end{equation}


Very few is known about $\mu$, except when $z$ belongs to the \emph{oriented percolation cone}.
This cone is defined as the set of points $z$ such that there exists with probability $1$ an infinite open path in the direction of $z$ taking only east/north edges.
For all $z$ in this cone, $\mu(z)$ is obviously equal to $|z|_1$.
Marchand \cite{Regine} showed that $\mu$ differs from $|.|_1$ outside this cone and previously Durrett \cite{PercoOrientee} had shown that this cone is delimited near the $x$-axis by a line $y=t_px$ where $t_p=1-p+\mathrm{o}(1-p)$.
Hence, we are interested in this paper in the remaining region $|y|\leq (1-p)x$.
\begin{theo}\label{TheoPrincipal}
For all $0\leq \lambda\leq 1$, on the event $\set{\Zero\leftrightarrow \infty}$, almost surely,
\begin{equation*}
\mu_p(\lambda):=\lim_{\substack{n \to \infty \\\Zero \leftrightarrow [n(1,\lambda(1-p))]}} \frac{D\left(\Zero,[n(1,\lambda(1-p))]\right)}{n}
=1+(1-p) \frac{1+\lambda^2}{2}+\mathcal{O}\left((1-p)^2\right).
\end{equation*}
\end{theo}


Let us also note that we actually obtain for all $p$ the following non-asymptotic lower bound for $\mu_p$, it is sharp when $p$ goes to one. It is a consequence of Corollary \ref{Coro:AsymptotiqueCrossModel} in Section \ref{Sec:Markov}.
\begin{theo}
For all $p>1/2$ and $0\leq \lambda\leq 1$,
$$
\mu_p(\lambda)\geq 2-\sqrt{1-(1-p)(1+\lambda^2)+\lambda^2(1-p)^2}.
$$
\end{theo}

Outside the cone, the exact limiting shape of large balls remains unknown.
Very recently, Auffinger and Damron \cite{Auffinger} showed that the corresponding limiting shape in first passage percolation is differentiable at the edge of the cone, thereby excluding the possibility of a polygon. It is believed that the limiting shape is strictly convex near axes, our result roughly says that, when $p$ is close to one, the four corners of $L^1$ balls are replaced by curves looking like arcs of parabola, as in the (schematic) figure below.
\begin{center}
\includegraphics[width=70mm]{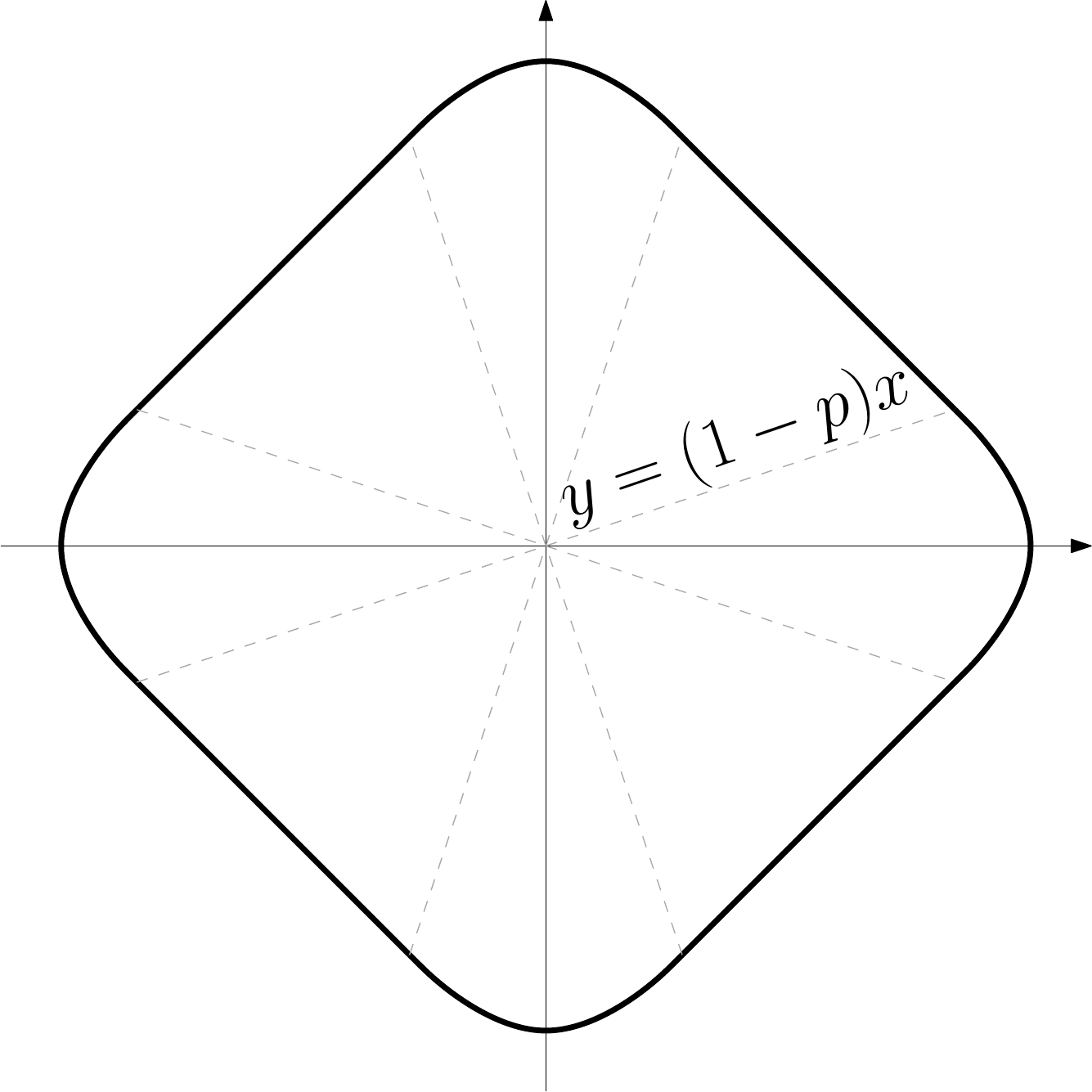}
\end{center}


The general strategy of the proof is based on a connection introduced in \cite{BEG} between the discrete-time \emph{totally asymmetric simple exclusion process} (TASEP) and distances on the percolation cluster. The TASEP is used by physicists as a simple model for nonequilibrium phenomena, it is known to be connected to a large class of combinatorial models : the corner-growth model, last passage percolation, random matrices,\dots (see \cite{SurveyRSK} for a survey). It seems that this connection with distances in the infinite percolation cluster appeared for the first time in \cite{BEG}.

The exact correspondence with TASEP on $\bbZ$ holds with a simplified model of percolation, which is described in Section \ref{Sec:Markov}. The lower bound for $\mu_p$ follows easily. For the upper bound, we first need a careful analysis of geodesics in the simplified model in order to show that they can be modified into an open path in the percolation cluster.

This strategy differs notably from that of \cite{BEG} where we worked in a large box around the $x$-axis, and thus the correspondence was with the TASEP on a finite interval. This restricted the analysis to points whose height was sublinear in $n$ and therefore gave results only for $\lambda=0$.
We also feel that this correspondence with TASEP is more transparent in the present paper than in \cite{BEG} and that it allows us to use more efficiently some known results on TASEP.


\section{The connection with TASEP}\label{Sec:Markov}

\subsection{Percolation in the cross model}

As a first step, we study a two-dimensional random graph in which distances to the origin behave much like distances in the infinite percolation cluster and are strongly connected the TASEP.

Here is the context we will deal with in the whole section. Let $\Cross$ be the graph on the vertices of $\Z^2$, with three kinds of edges:
\begin{itemize}
\item Vertical edges $\set{(i,j)\to(i,j+1),i\in \Z,j\in \Z}$;
\item Horizontal edges $\set{(i,j)\to(i+1,j),i\in \Z,j\in \Z}$;
\item Diagonal edges $\set{(i,j)\to(i+1,j+1)\text{ and }(i,j)\to(i+1,j-1)}$.
\end{itemize}
We assign length $1$ to each vertical and horizontal edge, and length $2$ to each diagonal edge.
We now set $\eps=1-p>0$ and call \emph{Cross Model} the percolation on $\Cross$ defined by:
\begin{itemize}
\item[\emph{(i)}] Diagonal and vertical edges are all open,
\item[\emph{(ii)}] Each horizontal edge is open (resp. closed) independently with probability $1-\varepsilon$ (resp. $\varepsilon$).
\end{itemize}
\begin{remark}
Let us first motivate this simplified model.
\begin{enumerate}
\item In classical percolation on $\bbZ^2$ with $\eps$ close to zero, a very large proportion (greater than $1-6\eps^2$) of unit squares of $\bbZ^2$ have at most one closed edge. In such squares, the addition of two diagonal edges of length $2$ does not change the time needed to cross the square from one corner to the other.
\item The opening of vertical edges should not be significant at first order since, as we will see later, a typical geodesic between $\Zero$ and $(n,n\lambda\eps)$ in $\Cross$ takes less than $2n\eps$ vertical edges, a proportion $\eps$ only of them being closed in the original model of percolation.
\end{enumerate}
\end{remark}

For $(i,j)\in\Z^2$, let $D^{\times}(i,j)$ be the distance between $(0,0)$ and $(i,j)$ in the Cross Model (see an example in Fig. 1).
Since vertical and diagonal edges are open, every point in $\Cross$ is connected in the Cross Model to $\Zero$, hence $D^{\times}(i,j)$ is finite for every $(i,j)$.
Let us write down some obvious consequences of the construction: for $i\geq 0$,
\begin{itemize}
\item All the geodesics joining $\Zero$ to $(i,j)$ only make N,NE,E,SE,S steps.
\item Along each vertical edge $|D^{\times}(i,j)-D^{\times}(i,j+1)|=1$.
\item Along each open horizontal edge, $D^{\times}(i+1,j)=D^{\times}(i,j)+1$.
\item Along each closed horizontal edge, $D^{\times}(i+1,j)=D^{\times}(i,j)+1$ or $+3$.
\item Along each diagonal edge, $D^{\times}(i+1,j\pm 1)=D^{\times}(i,j)+0$ or $+2$.
\end{itemize}
We also set $\mathbf{D}^{\times}_i$ for the (infinite) $i$-th column of distances $\set{D^{\times}(i,j),j\in \Z}$. Note that obviously
 $\mathbf{D}^{\times}_0=\left(\dots ,-2,-1,0,1,2,\dots \right)$.
The aim of the present section is to identify the law of the Markov chain $\left(\mathbf{D}^{\times}_i\right)_{i\geq 0}$.

\begin{figure}[h!]
\label{Fig:TASEP}
\begin{center}
\includegraphics[width=70mm]{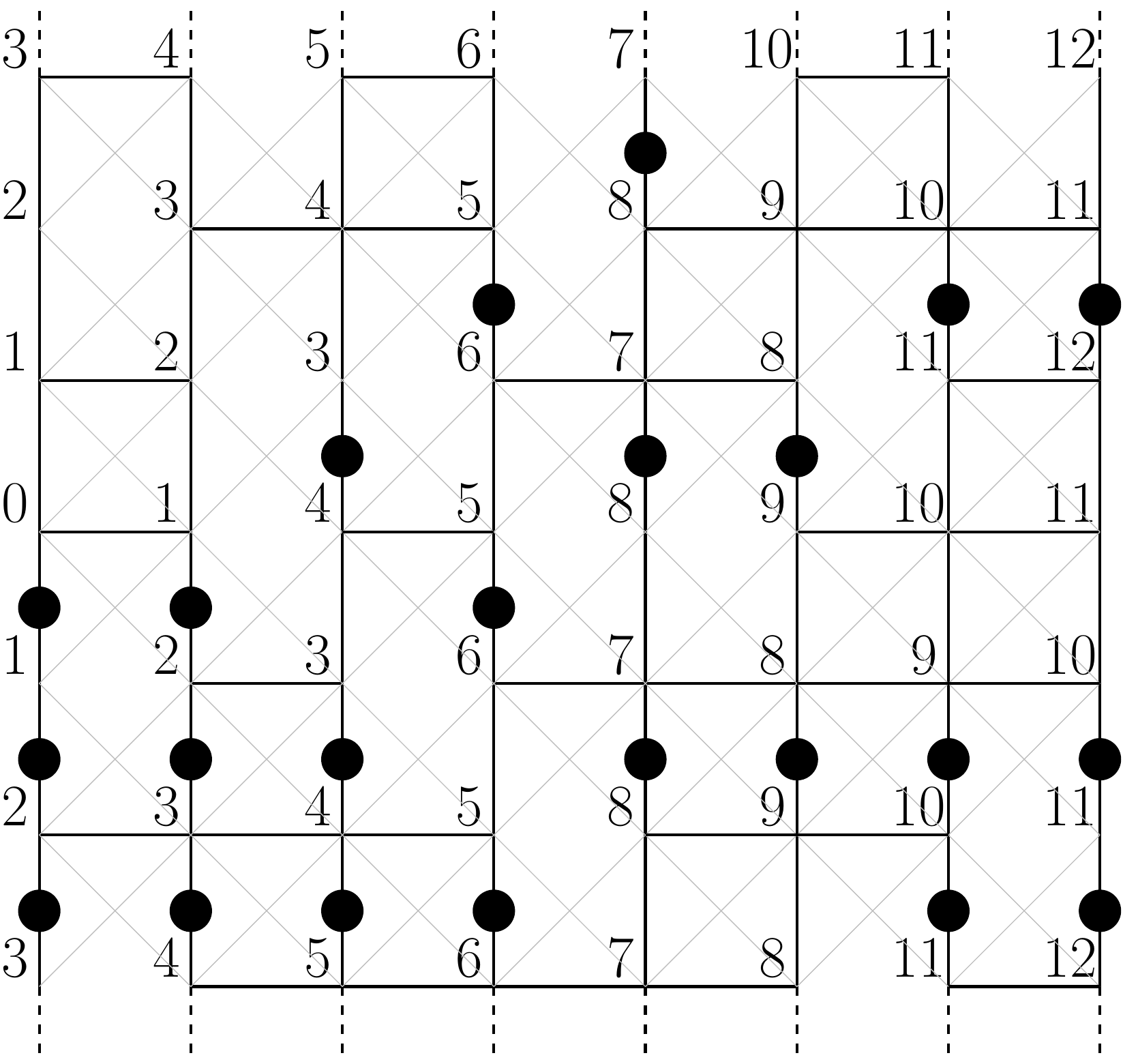}
\caption{A configuration of percolation in $\Cross$ with the associated distances $\mathbf{D}^{\times}_0,\dots \mathbf{D}^{\times}_7$, together with particles. Note the importance of diagonal edges: $D^\times(2,1)=3$ thanks to the diagonal edge $(1,0)\to(2,1)$.}
\end{center}
\end{figure}

To do so, we introduce a particle system closely related to the process $(\mathbf{D}^{\times}_i)_{i\ge 0}$.
Let us consider the state space $\set{\bullet,\circ}^{\Z}$ (identified to $\set{1,0}^{\Z}$), and denote its elements in the form
$$
(\dots,y^{-2},y^{-1},y^{0},y^1,y^2,\dots).
$$
Let $(\mathbf{Y}_i)_{i\geq 0}$ be the process with values in
$\set{\bullet,\circ}^{\Z}$  defined as follows :
\begin{equation*}\forall j\in \Z,\qquad
Y_i^j =
\begin{cases}
&\bullet =1 \text{ if } D^{\times}(i,j)= D^{\times}(i,j-1) -1.\\
&\circ =0  \text{ if } D^{\times}(i,j)= D^{\times}(i,j-1) +1.
\end{cases}
\end{equation*}
Let say that the site $j$ is \emph{occupied} by a particle at time $i$ if $Y^j_i=\bullet$ and \emph{empty} otherwise.
We think about a particle at site $j$ at time $i$ as being actually located on the edge $(i,j-1)-(i,j)$ as drawn on our pictures.


The main observation is that if we see time going from left to right, then the displacement of particles follows a discrete-time TASEP on $\Z$, that we define now:

\begin{definition}
The discrete-time \emph{Totally Asymmetric Simple Exclusion Process} (TASEP) on $\Z$ with parameter $\alpha$ is the Markov chain with state space $\set{\bullet,\circ}^{\Z}$ with initial condition $y_0$ defined by
$$
\dots,y^{-3}_0=y^{-2}_0=y^{-1}_0=y^0_0=\bullet,\qquad \circ=y^1_0=y^2_0=y^3_0=\dots
$$
and whose evolution is as follows: at time $t+1$, for each $j$, a particle at position $j$ (if any) moves one step forward if the site $j+1$ is empty at time $t$, with probability $\alpha$ and independently from the other particles.
\begin{center}
\includegraphics[width=50mm]{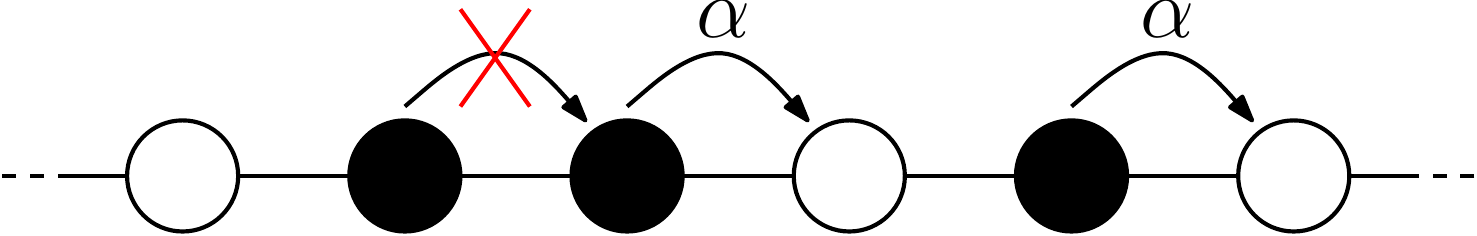}
\end{center}
\end{definition}

\begin{prop} The process $(\mathbf{Y}_i)_{i\geq 0}$ has the law of TASEP on $\bbZ$ with parameter $\eps$.
\end{prop}
Proposition 3 in \cite{BEG} states the same result in the framework of a finite interval of $\bbZ$, and the proof is identical.
Let us say however some words about it. The main point is that modifications in the particle configuration might only occur when a site $(i,j)$ at some distance $\ell$ lies between two sites of the line $\set{i}\times\bbZ$  which are at a distance $\ell +1$. 
In this case there is, at time $i$, a particle below $j$ and no particle above $j$. The particle below $j$ moves one step forward if and only if the horizontal edge $(i,j)\to(i+1,j)$ is closed (which occurs with probability $\eps$).

Let $J_{n,j}$ be the \emph{current} of the TASEP at time $n$ in $j$, that is the number of particles which have passed through position $j$ before time $n$:
$$
J_{n,j}:= \mathrm{card}\set{\ell> j, y_n^{\ell}=\bullet}.
$$

\begin{lem}\label{Lem:LienTASEP-PPP}
For each $n,j\geq 0$,
$$
D^\times(n,j)=n+j+2J_{n,j},
$$
where $J_{n,j}$ is the current of the TASEP with parameter $\eps$.
\end{lem}
\begin{proof}
Let us first prove this assertion for $j=0$. As already noticed, distances along each horizontal edge $(i,0)\to(i+1,0)$ differ from $1$ or $3$. This implies that
$$
D^\times(n,0)=n+2\ \card\set{0\leq i\leq n-1\ |\ D^\times(i+1,0)=D^\times(i,0)+3}.
$$
But $D^\times(i+1,0)=D^\times(i,0)+3$ occurs only in the case \raisebox{-10mm}{\includegraphics[width=22mm]{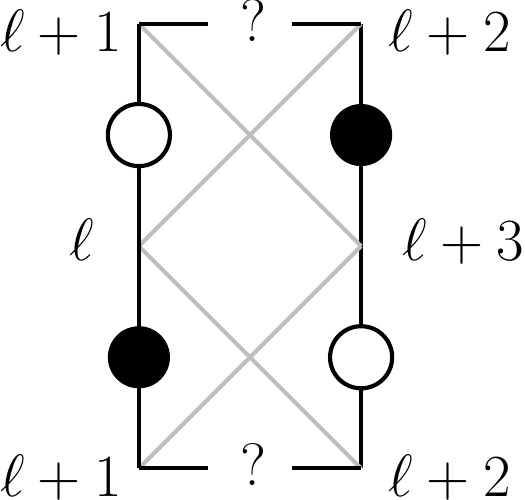}}, \emph{i.e.} when a particle jumps across the $x$-axis. Then $D^\times(n,0)\stackrel{\text{law}}=n+2J_{n,0}$.

To prove the Lemma for any $j\geq 0$, we have to compute $D^\times(n,j)-D^\times(n,0)$. But by construction of the particles
$$
D^\times(n,j)-D^\times(n,0) = j-2\ \card\set{\text{particles between $0$ and $j$ at time $n$}},
$$
and this finishes the proof.
\end{proof}

\subsection{Asymptotics in the cross model}

\begin{prop}\label{Prop:CourantTASEP}
For any $x>0$ et $-x\eps \leq y \leq x\eps$, we have almost surely and in $L^1$
$$
\frac{J_{\lf Nx\rf,\lf Ny\rf}}{N}\stackrel{N\to\infty}{\to} j(x,y):= \tfrac{1}{2}(x-y)-\tfrac{1}{2}\sqrt{(1-\eps)(x^2-y^2/\eps)}.
$$
\end{prop}
Note that if on the contrary $y>x\eps$, it is clear that $J_{\lf Nx\rf,\lf Ny\rf}/N\to 0$, because the right-most particle in the TASEP is at time $N$ at position $N\eps+\mathrm{o}(N)$.
\begin{proof}[Proof of Proposition \ref{Prop:CourantTASEP}]
With extra work this can be seen as a consequence of the work by Jockusch, Propp and Shor on the discrete-time TASEP (\cite{Domino}, Theorem 2). Here we deduce it from the results by Johansson \cite{Johansson} on last passage percolation (LPP) with geometric passage times with parameter $\eps$ (we refer to \cite{Sep} for the connection between discrete-time TASEP and last passage percolation).

For the reader's convenience, we detail the computations.
To define the model of LPP with geometric weights, let $(g_{i,j})_{i,j\geq 1}$ be i.i.d. geometric variables with parameter $\eps$. For a point $(i,j)$ in the quadrant $\set{i,j\geq 1}$, we write $G(x,y)$ for the \emph{last passage time} at $(x,y)$, \emph{i.e.}
$$
G^\star(x,y):= \max_{\gamma : 0\to(x,y)} \sum_{(i,j)\in \gamma} g_{i,j}
$$
where the $\max$ is taken over the $\binom{x+y-2}{x-1}$ paths with North/East steps going from $(1,1)$ to $(x,y)$.
Johansson (\cite{Johansson}, Theorem 1.1, see also \cite{Sep}, Theorem 2.2) has shown that for all $a,b>0$
$$
\frac{G^\star(\lf Na\rf,\lf Nb\rf )}{N} \to \Psi(a,b):= \frac{a+b+2\sqrt{(1-\eps)ab}}{\eps},
$$
where the convergence holds a.s. and in $L^1$
(note that the $p$ in Johansson's article, corresponds to $\eps=1-p$ with our notations).
Thanks to a plain correspondence between TASEP and LPP (see \cite{Sep} Proposition 1.2) there is coupling between LPP and discrete-time TASEP with parameter $\eps$ such that for any integers $A\geq B\geq 1$
\begin{equation}\label{Eq:FlotPPP}
J_{G^\star(A,B),A-B}=B.
\end{equation}
For $a\geq b >0$, let us write
$$
\frac{J_{\lf N\Psi(a,b)\rf,\lf N(a-b)\rf}}{N}
=
\frac{J_{G^\star(\lf Na\rf,\lf Nb\rf ), \lf N(a-b)\rf}}{N}
+\left(\frac{J_{\lf N\Psi(a,b)\rf,\lf N(a-b)\rf}}{N}
-\frac{J_{G^\star(\lf Na\rf,\lf Nb\rf ), \lf N(a-b)\rf}}{N}\right).
$$
Using \eqref{Eq:FlotPPP}, the first term in the right-hand side goes to $b$. The second term goes to zero almost surely since for any $n,n',j$ we have
$|J_{n,j}-J_{n',j}|\leq |n-n'|$.

We search $a=a(x,y)$ and $b=b(x,y)$ such that $\Psi(a,b)=x$ and $a-b=y$. This is possible if $y\leq x\eps$ and in this case we obtain
$$
\frac{J_{\lf Nx\rf,\lf Ny\rf}}{N}\stackrel{N\to\infty}{\to} b(x,y)
=\tfrac{1}{2}(x-y)-\tfrac{1}{2}\sqrt{(1-\eps)(x^2-y^2/\eps)}.
$$
\end{proof}

Taking $x=1,y=\eps\lambda$ in the Proposition, we obtain with Lemma \ref{Lem:LienTASEP-PPP} the following asymptotics for the distances in the cross model. Note that from now on, we skip the integer parts $[.]$ in order to lighten notations.
\begin{coro}\label{Coro:AsymptotiqueCrossModel}
In the cross model, for any $0<\eps,\lambda<1$,
$$
\frac{D^\times(n,n\eps\lambda)}{n} \stackrel{n\to\infty}{\to}
f(\lambda,\varepsilon):=2-\sqrt{1-\varepsilon(1+\lambda^2)+\lambda^2\varepsilon^2}
=1+\frac{\eps}{2}(1+\lambda^2) +\mathcal{O}(\eps^2),
$$
where the convergence is almost sure and in $L^1$.
\end{coro}

\section{The lower bound}

With the asymptotics found in the Cross Model, we are now able to obtain the lower bound for the distances in standard percolation in $\Z^2$.
Adding diagonal edges to $\Z^2$ decreases distances, so by an obvious coupling between percolation in $\Z^2$ and in the cross model we have
$$
\frac{D(\Zero,(n,n\eps\lambda))}{n}\geq \frac{D^\times(n,n\eps\lambda)}{n}.
$$
Letting $n$ go to infinity, we get
$$
\liminf_{n\to\infty} \frac{D(\Zero,(n,n\eps\lambda))}{n}\geq f(\lambda,\varepsilon)= 1+\frac{\eps}{2}(1+\lambda^2)+\mathcal{O}(\eps^2).
$$

\section{The upper bound}

The proof of the upper bound is more delicate. We first construct in a canonical way a geodesic in the Cross Model, and then show how to modify it to obtain an almost optimal path between $\Zero$ and $(n,n\eps\lambda)$ in the original model.

\subsection{The construction of a canonical geodesic}



Starting from the end $E:=(n,n\eps\lambda)$, we construct backwards a geodesic $\pi^\times$, in the Cross Model, joining $\Zero$ to $E$.
An important feature of this construction is that it only depends on the trajectory of the particles.

The reader is invited to follow the construction on the following example (here $E=(7,1)$ and $\pi^\times$ is drawn in red, $(\sigma_1,\sigma_2,\dots)$
stands for the sequence of particles ranked according to their height):
\begin{center}
\includegraphics[width=80mm]{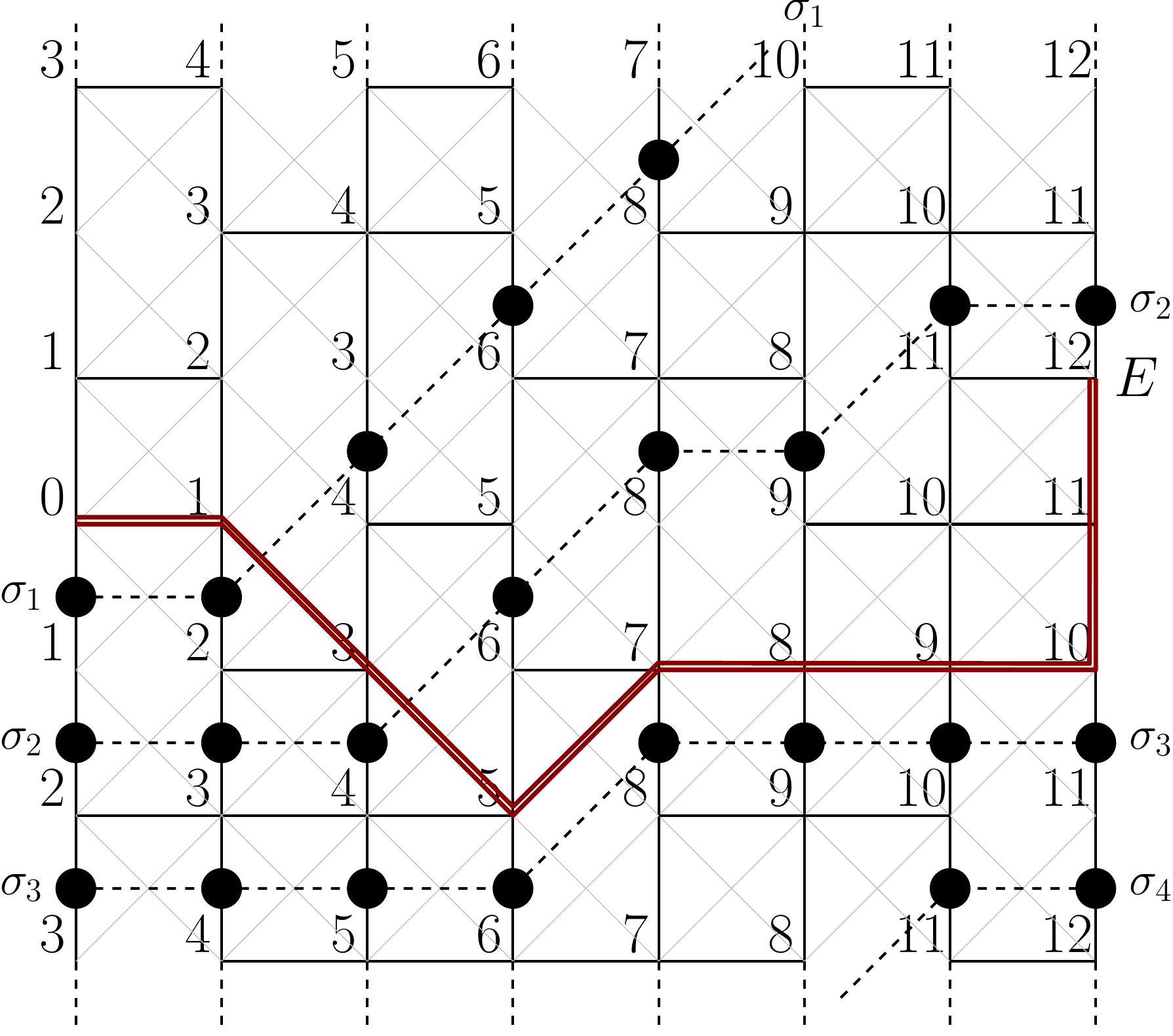}
\end{center}
The path starts (backwards) from $E$ by taking some vertical edges in the following way:
\begin{itemize}
\item if there is no particle on the vertical edge just below $E$ (as in the example), we go down until finding the first vertex that is just below an empty edge and just above an edge with a particle (in the example, until being at $(7,-1)$ just above particle $\sigma_3$);
\item if, on the contrary, there is a particle on the edge just below $E$, we go up until finding the first vertex that is below an empty edge and above an edge with a particle.
\end{itemize}
Note that if both conditions are realized, \emph{i.e.} if $E$ is just below an empty edge and just above an edge with a particle, then the path does not take any vertical edge.

We now proceed from right to left by taking $n$ horizontal or diagonal edges going to zero, so that each site of the path is just below an empty edge and just above a particle.
Let us write it more formally. After the first vertical edges, we are at a site with a certain particle $\sigma_p$ just below; let us denote by $(i,j)$ this site, and $\ell$ its distance to the origin.

Then, three cases may occur:
\begin{enumerate}
\item[{\bf Case A.}]
\begin{tabular}{m{95mm} m{25mm}}
Particle $\sigma_p$ had jumped at time $i-1$. Then the path follows the diagonal edge $(i,j)\to (i-1,j-1)$. Note that there is still an empty edge just above, if not $\sigma_p$ would not have moved.
&
\includegraphics[width=25mm]{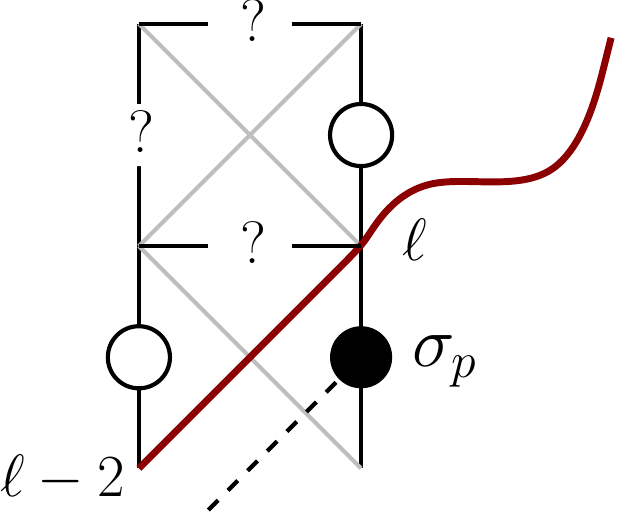}
\end{tabular}
\item[{\bf Case B.}]
\begin{tabular}{m{95mm} m{25mm}}
Particle $\sigma_{p-1}$ was just above $\sigma_p$ at time $i-1$. Then necessarily it moved (since edge $(i,j)\to(i,j+1)$ is now empty). The path follows  the diagonal edge $(i,j)\to (i-1,j+1)$.
&
\includegraphics[width=25mm]{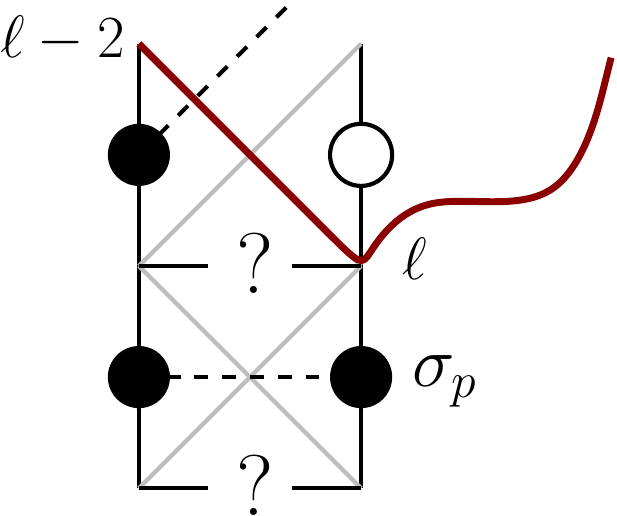}
\end{tabular}
\item[{\bf Case C.}]
\begin{tabular}{m{95mm} m{25mm}}
At time $i-1$, there is no particle above $\sigma_p$. This implies that $(i,j)\to(i-1,j)$ is open (if not, $\sigma_p$ would have moved).
The path follows this edge, and doing so it stays just above $\sigma_p$.
&
\includegraphics[width=25mm]{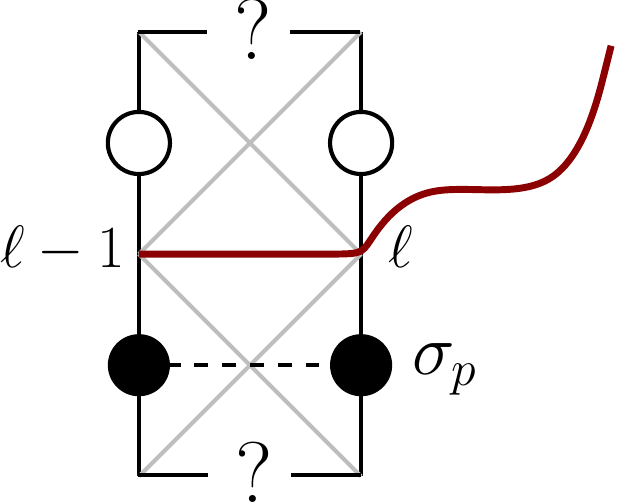}
\end{tabular}
\end{enumerate}

Let us record two features of this path $\pi^\times$:
\begin{itemize}
\item it takes only E,NE,SE edges until reaching $\set{x=n}$ (it takes exactly $n$ such steps), and then possibly taking some additional vertical edges in the form $(n,j)\to(n,j\pm 1)$ to reach $E$;
\item it takes a diagonal edge $(i,j)\to(i+1,j\pm1)$ only if the horizontal edge $(i,j)\to(i+1,j)$ is closed.
\end{itemize}

\begin{lem}\label{Lem:CheminSurLaVague}
The path $\pi^\times$ is a geodesic between $\Zero$ and $(n,n\eps\lambda)$ for the Cross Model. Moreover, $\pi^\times$ depends only on the trajectories of particles.
\end{lem}
\begin{proof}
The second assertion is clear by construction.
Besides, the path always goes through vertices which are just above a particle and below an empty edge. Thus, when the first coordinate is zero, it is necessarily at the origin, since this is the only site on the first column which satisfies this property.

Writing
$$
\pi^\times =\left(x_0=0,x_1,\dots,x_L=E\right),
$$
we have to prove that for each $i$ the length of the edge $(x_{i-1}, x_i)$ is equal to $D^\times(x_i)-D^\times(x_{i-1})$.

\begin{itemize}
\item By construction, if we had to take at the first stage $r$ vertical edges, these edges led to a site which is at distance $D^\times(n,n\eps\lambda)-r$ from the origin.
\item When the path takes an horizontal edge $(x_{i-1}, x_i)$ (case C above), this edge is open and then $D^\times(x_i)=D^\times(x_{i-1})+1$.
\item It remains the case of a diagonal edge $(x_{i-1}, x_i)$ (cases A,B above), we do the case $A$. Set $\ell=D^\times(x_i)$, since there is a particle on the edge $(x_i,x_i-(0,1))$, then $D^\times(x_i-(0,1))=\ell+1$. Since $\sigma_p$ has jumped then $D^\times(x_{i-1})=\ell+1-3$.
\end{itemize}
\end{proof}

\subsection{How to bypass bad edges}\label{Sec:Contour}
The aim of this section is to construct from the path $\pi^\times$ obtained by Lemma \ref{Lem:CheminSurLaVague} an open path of $\mathbb{Z}^2$ which is barely longer than $\pi^\times$.

Recall that $\pi^\times$ can take either horizontal, vertical or diagonal edges.
Since we want to construct an open path on $\mathbb{Z}^2$, we need to replace its diagonal edges and its final closed vertical edges by detours of open edges. 

We begin by doing a transformation which enables us to replace the diagonal edges of $\pi^\times$ without changing the length of the path.
If $\pi^\times$ takes a diagonal edge
 $(i,j)\rightarrow(i+1,j\pm 1)$ then we replace this edge by  the path $(i,j)\rightarrow(i,j\pm 1)\rightarrow(i+1,j\pm 1)$ :
\begin{center}
\includegraphics[width=75mm]{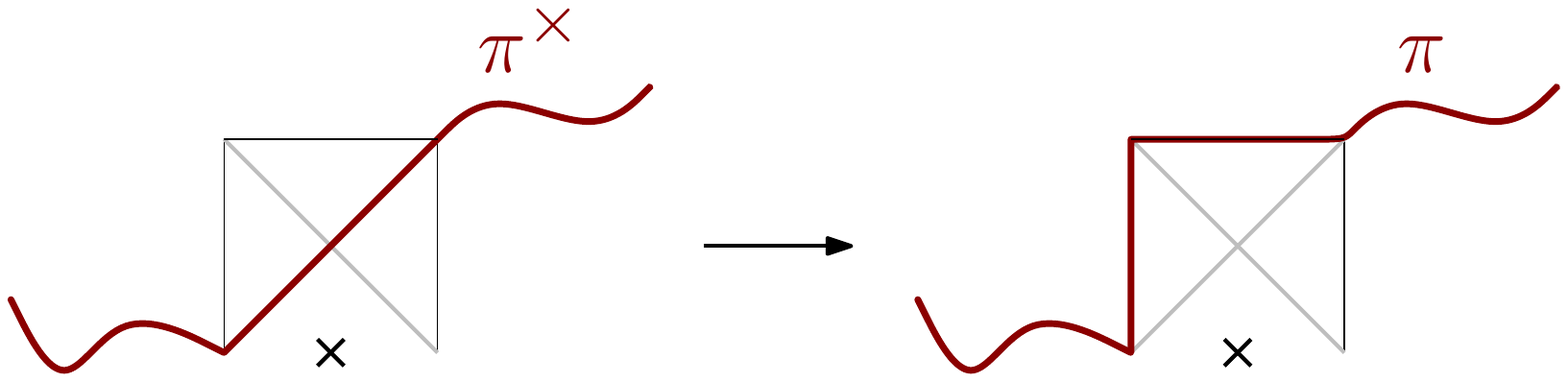}
\end{center}
We denote the new path by $\pi$.
We denote by $\mathcal{K}$ the set of edges which are either a vertical edge of $\pi^\times$ or an edge that appears in $\pi$ but not in $\pi^\times$.
Notice that $\pi^\times, \pi$ and $\mathcal{K}$ depends only on the TASEP.
The new path is a path on $\mathbb{Z}^2$ and it just remains to bypass its closed edges.
We call those closed edges the bad edges of $\pi$ and denote the set of all bad edges by $\mathcal{B}$.
By construction, $\mathcal{B}$ is a subset of $\mathcal{K}$.
We shall also write $K=|\mathcal{K}|$ and $B=|\mathcal{B}|$.

\begin{lem} \label{l:stic}
For $\eps$ small enough, for all $n$ large enough, we have $\mathbb{E}(K) \le 2n\eps$.
\end{lem}
\begin{proof}
The sum of the number of diagonal edges in $\pi^\times$ and of the number of final vertical edges, by definition of the cross model, is equal to
$
D^\times(n,n\eps\lambda)-n.
$
Corollary \ref{Coro:AsymptotiqueCrossModel} implies that $\mathbb{E}\left((D^\times(n,n\eps\lambda)-n)/n\right)$ converges to a limit which is strictly less that $\eps$ for small enough $\eps$.
The lemma follows from the fact that $K$ is at most $2(D^\times(n,n\eps\lambda)-n)$.
\end{proof}

Consider the dual graph $(\mathbb{Z}^2)^\star$ of $\mathbb{Z}^2$ and associate to each edge $e\in\Z^2$  the
unique edge $e^\star$ of the dual which crosses $e$. We say that $e^\star$ is open (resp. closed) if $e$ is open (resp. closed).
For each (closed) bad edge $e$,
consider the set $\mathcal{C}^\star(e)$ defined by
$$
\mathcal{C}^\star(e)=\{\hbox{closed edges of $(\mathbb{Z}^2)^\star$ connected to $e^\star$ by a path of closed edges}\}.
$$
Define its boundary
$$\Delta\mathcal{C}^\star(e)=\{\hbox{open edges of $(\mathbb{Z}^2)^\star$ which share at least one vertex with an edge of $\mathcal{C}^\star(e)$}\}.$$

\vspace{2mm}
\begin{center}
\begin{tabular}{m{50mm} m{70mm}}
What happens around a bad edge $e$: &
\includegraphics[width=80mm]{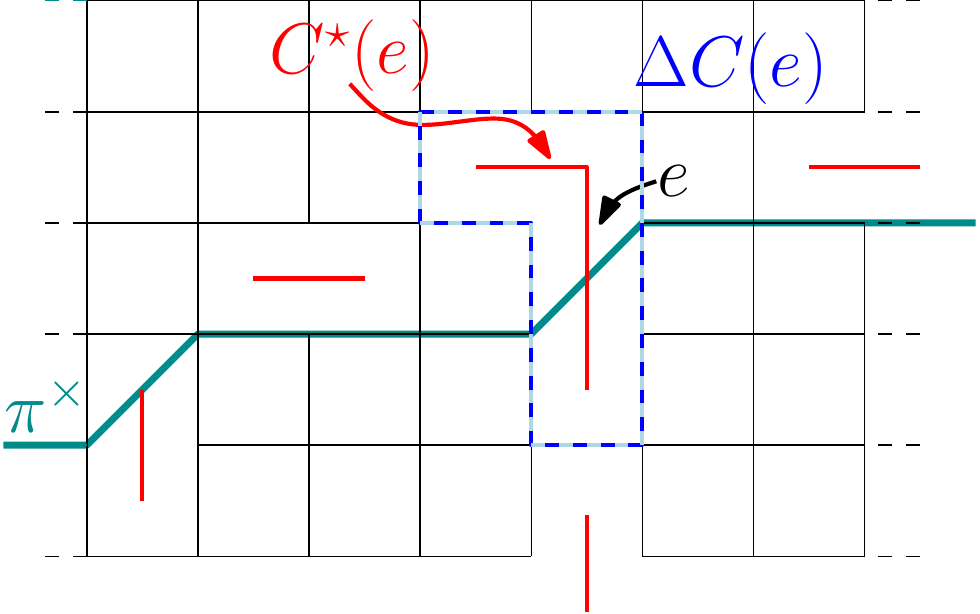}
\end{tabular}
\end{center}
\vspace{2mm}

We denote by $\Delta\mathcal{C}(e)$ the set of (open) edges associated to $\Delta\mathcal{C}^\star(e)$ in the initial graph, and set
$$
\Delta\mathcal{C}(\pi) = \bigcup_{e \in \mathcal{B}} \Delta\mathcal{C}(e).\\
$$

The following topological proposition will be crucial to build from $\pi^\times$ a short open path in the percolation cluster.

\begin{prop}\label{Claim}
On the event $\mathbf{0}\leftrightarrow (n,n\lambda\varepsilon)\leftrightarrow \infty$ there exists an open path in $\mathbb{Z}^2$ from $\mathbf{0}$ to $(n,n\lambda\varepsilon)$ which only uses edges of $\pi$ or of $\Delta\mathcal{C}(\pi)$.
\end{prop}

\begin{proof} The proof is purely topological and is postponed to Appendix.
\end{proof}

In view of this proposition, it remains to bound the cardinality of $\Delta\mathcal{C}(\pi)$ to get an upper bound of the length of the geodesic between $\Zero$ and $(n,n\lambda\varepsilon)$.
Let us describe the probability distribution of the set of open edges conditional on the position of particles.
An edge will be called \emph{unbiased} if, conditional on the position of the particles which are inherited from the cross-model, it is open with probability $1-\eps$ independently from all the states of the other edges.
An edge which is not unbiased is called \emph{biased}.

In the analysis of the state of an edge, four cases may arise.
\begin{itemize}
\item[{\bf Case 1.}] Vertical edges are all unbiased by definition of the cross model.
\item[{\bf Case 2.}] The two configurations of particles leading to an unbiased horizontal edge are the following:
\begin{center}
\includegraphics[height=25mm]{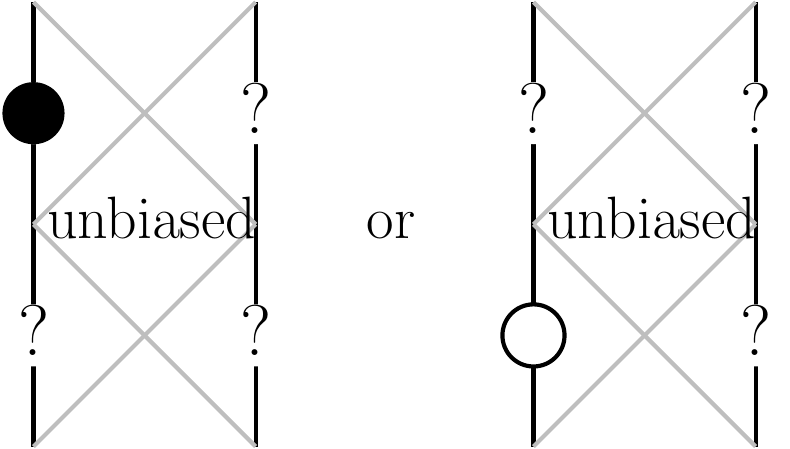}
\end{center}
Indeed, the state of the horizontal edge has no influence on the motion of particles.
\item[{\bf Case 3.}] The following configuration leads to a closed edge:
\raisebox{-10mm}{\includegraphics[height=25mm]{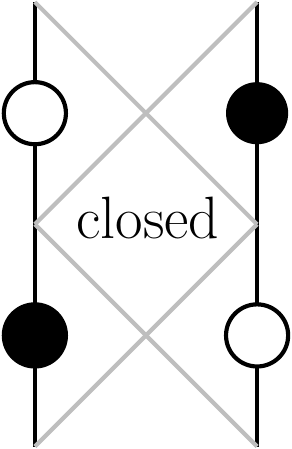}}
\item[{\bf Case 4.}] The following configuration leads to an open edge:
\raisebox{-10mm}{\includegraphics[height=25mm]{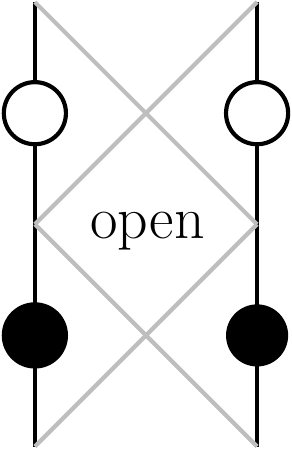}}
\end{itemize}

Denote by $\mathcal{T}$ the $\sigma$-field generated by the particles ($\mathcal{T}$ stands for TASEP).
Thus, conditional on $\mathcal{T}$, some edges are open, some edges are closed and the other edges are unbiased.
Now we use the previous facts to prove the following two lemmas:

\begin{lem}\label{Lem:MajoMauvaisCarres}
The edges of $\mathcal{K}$ are unbiased.
\end{lem}

\begin{proof}
Vertical edges of $\mathcal{K}$ are unbiased. This is Case $1$ above.
Let us now consider an horizontal edge of $\mathcal{K}$ associated with a rising diagonal edge of $\pi^\times$. By construction of $\pi^\times$ this corresponds to the following situation:
\begin{center}
\includegraphics[height=25mm]{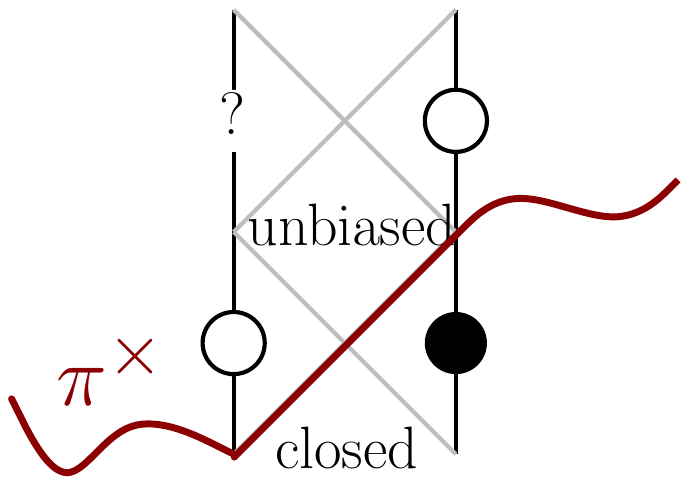}
\end{center}
This is the second situation in Case $2$ in the discussion just above.
The case of an horizontal edge of $\mathcal{K}$ associated with a downward diagonal of $\pi^\times$ corresponds to the first situation in Case $2$:
\begin{center}
\includegraphics[height=25mm]{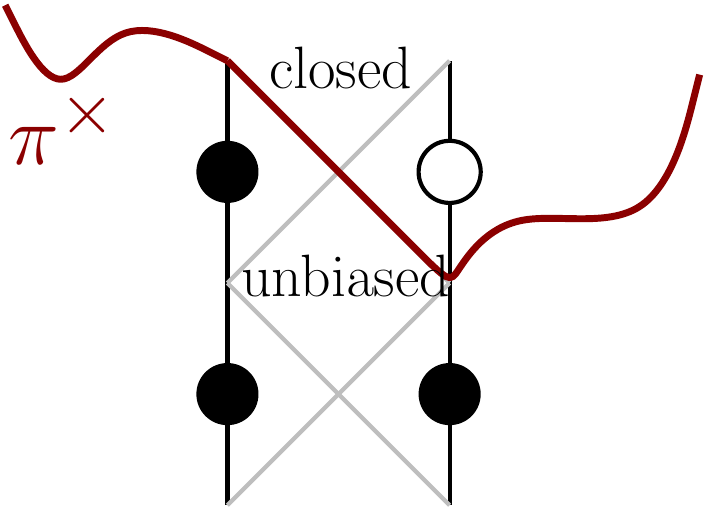}
\end{center}
\end{proof}

As a consequence of Lemma \ref{Lem:MajoMauvaisCarres}, conditional on $\mathcal{T}$, the number $B$ of bad edges is a binomial random variable with parameters $(K,\eps)$.
Using Lemma \ref{l:stic} we thus get that, for small enough $\eps$ and for large enough $n$, we have $\mathbb{E}(B) \le \mathbb{E}(K)\eps \le 2n\eps^2$.
The third item of the following lemma will enable us to bound the size of the detour associated with each bad edge $e \in \mathcal{B}$.
Items $1$ and $2$ are intermediate steps in the proof of Item $3$.

We say that $e^\star$ in $(\Z^2)^\star$ is unbiased if its dual edge $e$ is.

\begin{lem}\label{LemmeNath}
\begin{enumerate}
\item If an edge $e^\star$ is biased, its six neighbouring edges are unbiased.
\item If an animal $\mathcal{A}$ of $(\Z^2)^\star$ (\emph{i.e.} a connected component of edges in $(\Z^2)^\star$)
contains an edge of $\mathcal{K}$, then it contains at least $\max(1,|\mathcal{A}|/7)$ unbiased edges.
\item Conditional on $\mathcal{T}$, if $e$ belongs to $\mathcal{K}$, then
$$
\mathbb{E}\left(|C^\star(e)| 1_\mathcal{B}(e) \Big| \mathcal{T}\Big.\right) \le C\eps
$$
for small enough $\eps$ where $C$ is an absolute constant.
\end{enumerate}
\end{lem}
\begin{proof}
Item $1$. Take a biased edge $e^\star$. It corresponds either to Case $3$ above either to Case $4$.
As the proofs are identical, let us consider Case $3$.
It corresponds to the following situation, where we draw in red the dual edges of the six neighbouring edges of $e^\star$:
$$
\includegraphics[height=25mm]{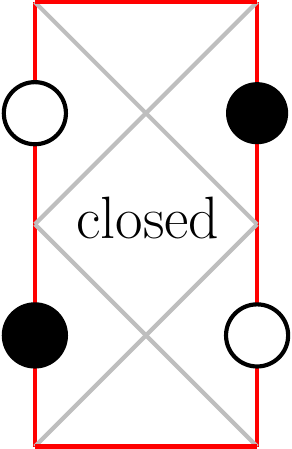}.
$$
Four of them are vertical and thus unbiased.
The horizontal edge above corresponds to the second situation of Case 2, whereas the horizontal edge below corresponds to the first situation of Case 2.

Item $2$. Let $\mathcal{A}$ be an animal of $(\Z^2)^\star$ and assume that $\mathcal{A}$ contains an edge of $\mathcal{K}$.
As all edges of $\mathcal{K}$ are unbiased by Lemma \ref{Lem:MajoMauvaisCarres},
the required lower bound is straightforward if the cardinality of $\mathcal{A}$ is $1$.
Let us assume that the cardinality of $\mathcal{A}$ is at least $2$.
By Item $1$, we can then construct an application which associates to a biased edge of $\mathcal{A}$ one of its unbiased neighbours in $\mathcal{A}$.
At most $6$ biased edges are mapped to the same unbiased edge (this is not optimal, one could replace $6$ by $2$).
Thus, $\mathcal{A}$ contains at least $\mathcal{A}/7$ unbiased edges.

Item $3$.
Let us condition on ${\cal T}$.
Let $e \in \mathcal{K}$.
There are less than $15^k$ animals of cardinality $k$ containing the edge $e^\star$.
(This can be proven for instance by adapting slightly the arguments of (4.24) p.81 in \cite{Grimmett}.)
Using Item 2, we thus get:
\begin{eqnarray*}
\mathbb{E}\left(|C^\star(e)| 1_\mathcal{B}(e) \Big| \mathcal{T}\Big.\right)
 & = & \sum_{k \ge 1} \mathbb{P}\left( |C^\star(e)| 1_\mathcal{B}(e) \ge k \Big| \mathcal{T}\Big.\right) \\
 & \le & \sum_{k \ge 1} 15^k \eps^{\max(1,k/7)} \\
 & \le & C\eps
\end{eqnarray*}
for an absolute constant $C$ and for small enough $\eps$.
\end{proof}

We are now in position to conclude the proof of Theorem \ref{TheoPrincipal} by giving an upper bound for $D(\mathbf{0},(n,n\lambda\varepsilon))$.

\begin{proof}[Proof of Theorem \ref{TheoPrincipal}]
Since an edge in $\mathcal{C}^\star(e)$ is connected to less than $6$ open edges,
$$
|\Delta\mathcal{C}^\star(e)|\le 6|\mathcal{C}^\star(e)|.
$$
Therefore,
$$
|\Delta\mathcal{C}(\pi)| \le \sum_{e \in \mathcal{B}} |\Delta\mathcal{C}(e)|
= \sum_{e \in \mathcal{B}} |\Delta\mathcal{C}^\star(e)|
\le 6\sum_{e \in \mathcal{B}} |\mathcal{C}^\star(e)|.
$$
Using Lemma \ref{LemmeNath} and Lemma \ref{l:stic} we thus get, for small enough $\eps$ and large enough $n$,
\begin{equation}\label{e:ah}
\mathbb{E}\left(|\Delta\mathcal{C}(\pi)| \right)
\le 6 \mathbb{E}\left(\sum_{e \in \mathcal{K}} \mathbb{E}\left(|C^\star(e)| 1_\mathcal{B}(e) \Big| \mathcal{T}\Big.\right) \right)
\le 6 C\eps \mathbb{E}(K) \le 12Cn\eps^2.
\end{equation}

Set $\mathcal{E}_n=\set{\mathbf{0}\leftrightarrow (n,n\lambda\varepsilon)\leftrightarrow \infty}$.
We deduce from Proposition \ref{Claim} that on the event $\mathcal{E}_n$,
$$
D(\mathbf{0},(n,n\lambda\varepsilon))\le |\pi|+|\Delta\mathcal{C}(\pi)| = D^\times(n,n\lambda\varepsilon)+|\Delta\mathcal{C}(\pi)|.
$$
This yields, for all $A>0$,
\begin{multline*}
\mathbb{P}(D(\mathbf{0},(n,n\lambda\varepsilon))\ge nf(\lambda,\eps) + An\eps^2,\mathcal{E}_n) \\
\leq \mathbb{P}(D^\times(n,n\lambda\varepsilon)\ge nf(\lambda,\eps) + An\eps^2/2) + \mathbb{P}(|\Delta\mathcal{C}(\pi)| \ge An\eps^2/2) \\
\leq \mathbb{P}(D^\times(n,n\lambda\varepsilon)\ge nf(\lambda,\eps) + An\eps^2/2) + \frac{\mathbb{E}(|\Delta\mathcal{C}(\pi)|)}{An\eps^2/2}.
\end{multline*}
From Corollary \ref{Coro:AsymptotiqueCrossModel} and \eqref{e:ah} we get, for small enough $\eps$,
$$
\limsup_{n \to \infty} \mathbb{P}(D(\mathbf{0},(n,n\lambda\varepsilon))\ge nf(\lambda,\eps) + An\eps^2,\mathcal{E}_n) \le 24C/A.
$$
Note that for large $n$, $\mathbb{P}(\mathcal{E}_n)\geq 1/2$. Since we know that $D(\mathbf{0},(n,n\lambda\varepsilon))/n$ converges almost surely to the constant $\mu_{1-\eps}(\lambda)$, we get, for $A>48C$,
$$
\mu_{1-\eps}(\lambda) \le f(\lambda,\eps)+ An\eps^2.
$$
\end{proof}


\subsection*{Appendix: Proof of Proposition \ref{Claim}}
\begin{proof}
Let $e_1$ be an edge in $\mathcal{B}$, we denote  by $(\Delta \mathcal{C}(e_1))_\infty$ the set of edges $e$ of $\Delta \mathcal{C}(e_1)$ for which there exists a path in $\mathbb{Z}^2$ (which may take open or closed edges) such that
\begin{itemize}
\item the edge $e$ is the first edge of the path,
\item the path does not cross $\mathcal{C}^\star(e_1)$,
\item the path goes to infinity.
\end{itemize}
In other words, $(\Delta \mathcal{C}(e_1))_\infty$ is composed of the edges of $\Delta \mathcal{C}(e_1)$ that are not in the interior of $\mathcal{C}^\star(e_1)$.

\vspace{2mm}
\begin{center}
\includegraphics[width=95mm]{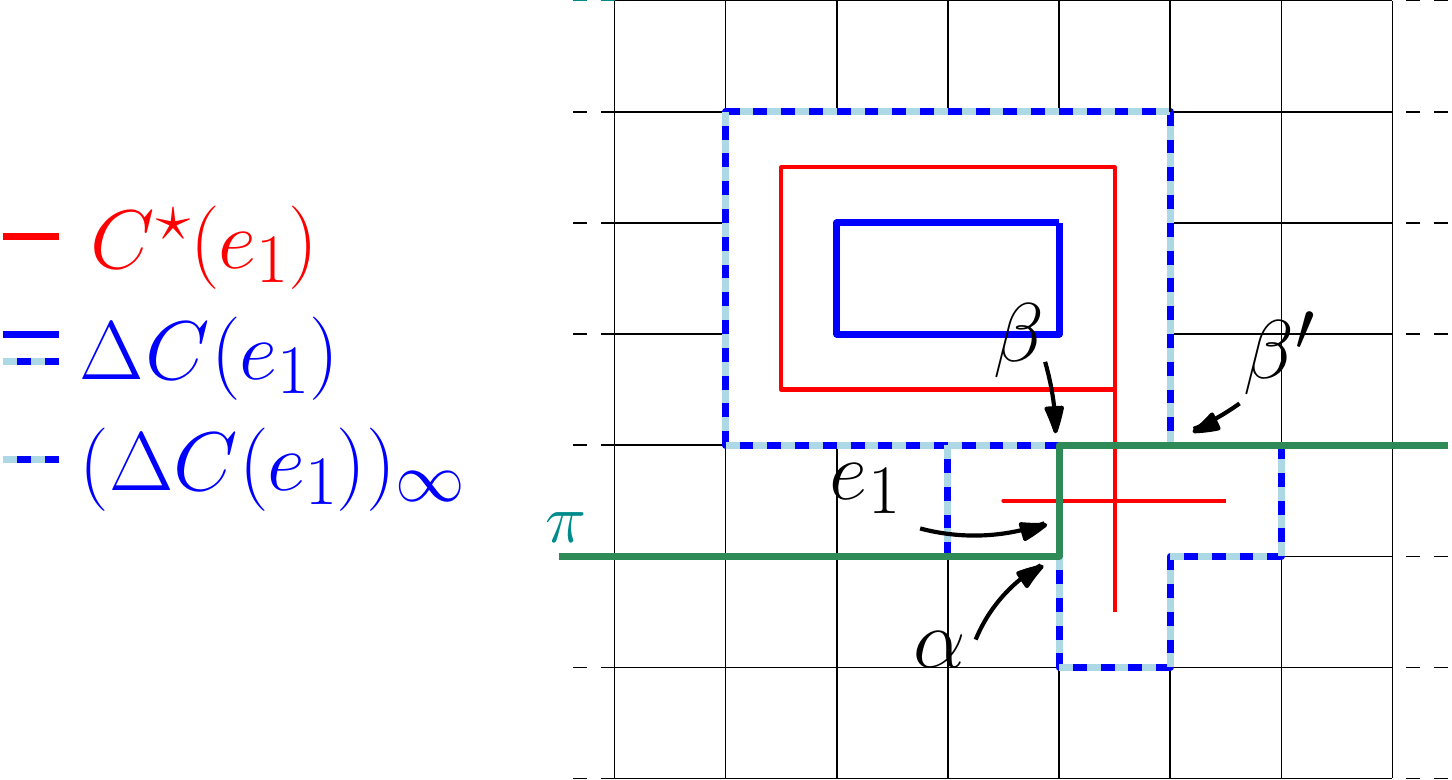}
\end{center}
\vspace{2mm}

A first observation is that the set $(\Delta \mathcal{C}(e_1))_\infty$ is connected. Take it for granted for a while and let us prove that there exists a path from $\mathbf{0}$ to $(n,n\lambda\varepsilon)$ which only uses edges of $\pi\setminus \set{e_1}$ or edges of $(\Delta\mathcal{C}(e_1))_\infty$. Let $(\alpha,\beta)$ be the first edge of $\pi$ crossing $\mathcal{C}^\star(e_1)$. One can check that $\alpha$ is necessarily the extremity of an edge of $(\Delta \mathcal{C}(e_1))_\infty$. Following the path $\pi$ backwards from $(n,n\lambda\varepsilon)$, we define similarly  another site  $\beta'\in (\Delta \mathcal{C}(e_1))_\infty$. The origin $\mathbf{0}$ is connected to $\alpha$ by $\pi$, $\alpha$ is connected to $\beta'$ by $(\Delta \mathcal{C}(e_1))_\infty$, and $\beta'$ is connected to $(n,n\lambda\varepsilon)$ by $\pi$. We proceed recursively for the next edges of $\mathcal{B}$.

It remains to prove that $(\Delta \mathcal{C}(e_1))_\infty$ is connected.
Let $\Sigma(\mathcal{C}^\star(e_1))\in \mathbb{Z}^2$ be the circuit surrounding $\mathcal{C}^\star(e_1)$ as defined in Proposition 11.2 of \cite{Grimmett}. By construction, $\Sigma(\mathcal{C}^\star(e_1))\subset
(\Delta \mathcal{C}(e_1))_\infty$. Assume that $(\Delta \mathcal{C}(e_1))_\infty$ is  not connected and let $H$ be a connected component which does not contain $\Sigma(\mathcal{C}^\star(e_1))$. The set $H$ being a connected component of $(\Delta \mathcal{C}(e_1))_\infty$, an edge in $\Delta H$ can not be in $\Delta \mathcal{C}(e_1)$.  This implies in particular that either
\begin{itemize}
\item[(i)]$\Sigma(H)\subset \mathcal{C}^\star(e_1)$
\item[(ii)] $\Sigma(H)$ and $\mathcal{C}^\star(e_1)$ have no site in common.
\end{itemize}
The first possibility contradicts the connection of $H$ with infinity, whereas the second one contradicts the connectivity of $\mathcal{C}^\star(e_1)$.
\end{proof}

\noindent {\bf Aknowledgements.} The two first authors are glad to thank ANR Grant \textsc{ Mememo 2} for the support.


\noindent \textsc{Anne-Laure Basdevant} \verb|anne-laure.basdevant@u-paris10.fr|,\\
\textsc{Nathana\"el Enriquez} \verb|nenriquez@u-paris10.fr|, \\
\textsc{Lucas Gerin} \verb|lgerin@u-paris10.fr|\\
Universit\'e Paris-Ouest Nanterre\\
Laboratoire Modal'X\\ 200 avenue de la R\'epublique\\
92000 Nanterre (France).

\vspace{5mm}

\noindent \textsc{Jean-Baptiste Gou\'er\'e} \verb|jean-baptiste.gouere@univ-orleans.fr|\\
MAPMO - F\'ed\'eration Denis Poisson\\
Universit\'e d'Orl\'eans\\
B.P. 6759 - 45067 Orl\'eans (France)

\end{document}